\documentclass[12pt]{amsart}
\usepackage{fullpage,amssymb}
\newtheorem{theorem}{Theorem}[section]
\newtheorem{lemma}[theorem]{Lemma}
\newtheorem{corollary}[theorem]{Corollary}
\newtheorem{proposition}[theorem]{Proposition}
\newtheorem*{notation*}{Notation}
\newtheorem*{p*}{Proposition~\ref{h.s.o.p}}
\theoremstyle{definition}
\newtheorem{definition}[theorem]{Definition}

\newtheorem{remark}[theorem]{Remark}
\newcommand{\M}{{\operatorname{Mat}}}

\newcommand{\N}{{\mathbb Z}_{\geq 0}}
\newcommand{\Z}{{\mathbb Z}}

\newcommand{\K}{K}
\newcommand{\Sl} {{\operatorname{SL} }}

\newcommand{\SL}{{\operatorname{SL}_n \times \operatorname{SL}_n}}

\newcommand{\rank}{\operatorname{rank}}

\newcommand{\Hom}{\operatorname{Hom}}
\newcommand{\kar}{\operatorname{char}}

\newcommand{\llangle}{\,\,(\!\!\!\!<\!}
\newcommand{\rrangle}{\!>\!\!\!\!)\,\,}
\newcommand{\Q}{{\mathbb Q}}
\newcommand{\GL}{\operatorname{GL}}
\newcommand{\Rep}{\operatorname{Rep}}
\newcommand{\SI}{\operatorname{SI}}
\newcommand{\lcm}{\operatorname{lcm}}
\newcommand{\Proj}{\operatorname{Proj}}

\usepackage{multirow}
\usepackage{amsmath}
\usepackage{comment}
\usepackage{tikz}
\usepackage{mathtools}
\usepackage{arydshln}

\title{Polynomial degree bounds for matrix semi-invariants}

\author{Harm Derksen and Visu Makam}

\thanks{The first author was supported by NSF grant DMS-1302032 and the second author was supported by NSF grant DMS-1361789}

\begin{document}

\begin{abstract}
We study the left-right action of $\SL$ on $m$-tuples of $n \times n$ matrices with entries in an infinite field $\K$. We show that invariants of degree $n^2- n$ define the null cone. Consequently,  invariants of degree $\leq n^6$ generate the ring of invariants if $\kar(K)=0$. We also prove that for $m \gg 0$,  invariants of degree at least $n\lfloor \sqrt{n+1}\rfloor$ are required to define the null cone. We generalize our results to matrix invariants of  $m$-tuples of $p\times q$ matrices,
and to rings of semi-invariants for quivers. For the proofs, we use new techniques such as the regularity lemma by Ivanyos, Qiao and Subrahmanyam, and the concavity property of the tensor blow-ups of matrix spaces. We will discuss several applications to algebraic complexity theory, such as a deterministic polynomial time algorithm for non-commutative rational identity testing, and the existence of small division-free formulas for non-commutative polynomials.
\end{abstract}

\maketitle

\section{Introduction} \label{intro}
\subsection{Degree bounds for invariant rings}
Let $\M_{p,q}$ be the set of $p\times q$ matrices with entries in an infinite field $K$.  The group $\GL_n$ acts on $\M_{n,n}^m$ by simultaneous conjugation.
Procesi showed that in characteristic 0, the invariant ring is generated by traces of words in the matrices.  Razmyslov (\cite[final remark]{Raz}) showed that that the invariant ring is generated by polynomials of degree $\leq n^2$ by studying trace identities (see also~\cite{Formanek}). In positive characteristic, 
generators of the invariant ring were given by Donkin in \cite{Donkin1,Donkin2}.  Domokos proved an upper bound $O(n^7m^n)$ for the degree of generators (see~\cite{Dom02,DKZ}). 

In this paper we will focus on the left-right action of $G=\SL$ on the space $V=\M_{n,n}^m$ of $m$-tuples of $n\times n$ matrices. This action is given by
$$
(A,B)\cdot (X_1,X_2,\dots,X_m)=(AX_1B^{-1},AX_2B^{-1},\dots,AX_mB^{-1}).
$$
The group $G$ also acts on the graded ring $K[V]$ of polynomial functions on $V$, and the subring of $G$-invariant polynomials is denoted by
$R(n,m)=K[V]^G$. This subring inherits a grading $R(n,m)=\bigoplus_{d=0}^\infty R(n,m)_d$.  We have $R(n,m)_d=0$ unless $d$ is divisible by $n$ (see Theorem~\ref{theo:gens}). It is well-known that $R(n,1)$ is generated by the determinant $\det(X_1)$, and $R(n,2)$ is generated by the coefficients of $\det(X_1+tX_2)$ as a polynomial in $t$. Because the group $G$ is reductive, this invariant ring is finitely generated (see~\cite{Hilbert1, Hilbert2, Nagata, Haboush}).
\begin{definition}
The number $\beta(n,m)$ is the smallest nonnegative integer $d$ such that $R(n,m)$ is generated by invariants of degree $\leq d$.
\end{definition}

The following bounds are known if $K$ has characteristic $0$:
\begin{enumerate}
\item $\beta(n,1)=\beta(n,2)=n$;
\item $\beta(1,m)=1$;
\item $\beta(2,m)\leq 4$;
\item $\beta(3,3) = 9$;
\item $\beta(3,m)\leq 309$;
\item $\beta(n,m)\geq n^2$ if $m\geq n^2$;
\item $\beta(n,m)=O(n^4((n+1)!)^2)$.
\end{enumerate}
The bounds in (1) follow from the descriptions of $R(n,1)$ and $R(n,2)$ above and (2) is trivial. The bound (3) can be found in \cite{Dom00a} (see also~\cite{IQS2}). This bound also follows from the First Fundamental Theorem of Invariant Theory
for $\operatorname{SO}_4$, because $\Sl_2\times \Sl_2$ is a finite central extension of $\operatorname{SO}_4$ and
the representation $\M_{2,2}$ of $\Sl_2\times \Sl_2$ corresponds to the standard $4$-dimensional representation of
  $\operatorname{SO}_4$. The bound (4) was given in \cite{Dom00b}.
  (5) and (6) were proved by the second author in \cite{Visu}. Some explicit upper bounds for $\beta(3,m)$
  for $m=4,5,6,7,8$ that are sharper than (5) were also given in \cite{Visu}.
  For the ring of invariants of a rational representation of a reductive group, there is a general bound on the degree of generating invariants (see \cite{Derksen1} and \cite[Section~4.7]{DK}). This bound gives $O(n^816^{n^2})$ and Ivanyos, Qiao and Subrahmanyam showed in \cite{IQS,IQS2} that this bound can be improved to (7). We will improve this factorial bound to a polynomial one:
\begin{theorem}\label{deg.bds}
If $K$ has characteristic $0$, then we have $\beta(n,m)\leq mn^4$.
\end{theorem}
A theorem of Weyl (see \cite[Section 7.1, Theorem A]{KP}) essentially tells us that a bound on the degree of generating invariants for $R(n,n^2)$ will be a bound on the degree of generating invariants for $R(n,m)$ for all $m$.  So we have:
\begin{corollary}
If $K$ has characteristic $0$, then we have $\beta(n,m)\leq n^6$.
\end{corollary}
Given two matrices $A = (a_{ij})$ of size $m \times n$, and $B = (b_{ij})$ of size $p \times q$, we define their tensor (or Kronecker) product to be 
$$A \otimes B = \begin{bmatrix}
a_{11}B   & a_{12}B     &  \cdots     & a_{1n}B  \\
 a_{21}B  &   \ddots               &                 & \vdots  \\
 \vdots      &       &    \ddots  &\vdots   \\
 a_{m1}B  &       \cdots       & \cdots               & a_{mn}B \\
 \end{bmatrix}\in \M_{mp,nq}.$$
For $T=(T_1,T_2,\dots,T_m)\in \M_{d,d}^m$, we define an invariant $f_{T}\in R(n,m)$ of degree $dn$ by
$$
f_T(X_1,X_2,\dots,X_m)=\det(X_1\otimes T_1+X_2\otimes T_2+\cdots+X_m\otimes T_m).
$$

Consider the generalized Kronecker quiver $\theta(m)$ which is a graph with 2 vertices $x$ and $y$ and $m$ arrows from $x$ to $y$.
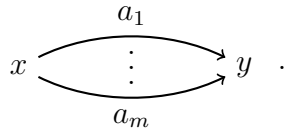
\begin{center}
\begin{tikzpicture}

\path node (x) at (0,0) []{$x$};
\path node (y) at (3,0) []{$y$};

\draw [->,thick] (x) .. controls (1,0.5) and (2,0.5) ..
node[midway,above] {$a_1$}
(y);

\draw [->,thick] (x) .. controls (1,-0.5) and (2,-0.5) ..
node[midway,below] {$a_m$}
(y);

\path node at (1.5,0.2) [] {.};
\path node at (1.5,0) [] {.};
\path node at (1.5,-0.2) [] {.};

\node at (3.5,0) []{$.$};
\end{tikzpicture}
\end{center}
Then $R(n,m)$ is the ring of semi-invariants for the quiver $\theta(m)$ and dimension vector $(n,n)$. 
Generators of $R(n,m)$ can be given in terms of determinants of certain block matrices, see \cite[Corollary 3]{DW},  \cite{DZ} and \cite{SVd}.
These results, applied to the Kronecker quiver $\theta(m)$ give:
\begin{theorem}\label{theo:gens}
The invariant ring $R(n,m)$ is generated by all  $f_T$ with $T\in \M_{d,d}^m$ and $d\geq 1$.
\end{theorem}
\subsection{Hilbert's null cone}
Hilbert's null cone ${\mathcal N}={\mathcal N}(n,m)\subseteq V$ is the zero set of all non-constant homogeneous invariants in $R(n,m)$. 
The null cone plays an important role in Geometric Invariant Theory.
\begin{definition} We define the constant $\gamma(n,m)$ as the smallest positive integer $d$
such that the non-constant homogeneous invariants of degree $\leq d$ define the null cone.
\end{definition}
 If $K$ has characteristic 0, then polynomial bounds for $\gamma(n,m)$ imply polynomial bounds for $\beta(n,m)$ (see~\cite{Derksen1}). 
  From the description of the invariants in Theorem~\ref{theo:gens} follows:
 \begin{corollary}
 The following statements are equivalent:
 \begin{enumerate}
\item $X=(X_1,X_2,\dots,X_m)$ does not lie in the null cone ${\mathcal N}(n,m)$;
\item $f_T(X)\neq 0$ for some $T\in \M_{d,d}^m$ with $d\geq 1$.
\end{enumerate}
\end{corollary}
\begin{definition}
Let $\delta(n,m)$ be the smallest positive integer $k$ such that $X\not\in {\mathcal N}(n,m)$ implies that there exists an integer $d$ with $1\leq d\leq k$
and an $m$-tuple $T=(T_1,\dots,T_m)\in \M_{d,d}^m$ of $d\times d$ matrices, such that $f_T(X)\neq 0$.
Since the $f_T$'s generate the invariant ring by Theorem~\ref{theo:gens}, it is clear that $\gamma(n,m)=n\delta(n,m)$.
\end{definition}

\begin{theorem}\label{main}
If $n\geq 2$, $X=(X_1,\dots,X_m)\notin {\mathcal N}(n,m)$ and $d\geq n-1$, then there exists an $m$-tuple $T\in \M_{d,d}^m$ such that $f_T(X)\neq 0$. In particular $\delta(n,m)\leq n-1$ and $\gamma(n,m)\leq n(n-1)$.
\end{theorem}

\begin{lemma}  \label{gamman}
The function  $\delta(n,m)$ is a weakly increasing function of $m$, and
for $m>n^2$ we have $\delta(n,m)=\delta(n,n^2)$.
\end{lemma}
Let us define $\delta(n)=\max_m \delta(n,m)=\delta(n,n^2)$ and $\gamma(n)=\gamma(n,n^2)=n\delta(n)$.
We prove a  lower bound on $\delta(n)$ which indicates  that the upper bound we find in Theorem~\ref{main} is quite strong.

\begin{theorem} \label{lower bounds}
We have $\delta(n) \geq  \lfloor\sqrt{n+1}\rfloor$ and $\gamma(n)\geq n\lfloor\sqrt{n+1}\rfloor$.
\end{theorem}
\subsection{Degree bounds for rings of quiver semi-invariants}
For details and notational conventions we refer to Section~\ref{sec:Quivers}.
To a quiver $Q$ with vertex set $Q_0$, and a dimension vector $\alpha\in \N^{Q_0}$ one can associate a ring  $\SI(Q,\alpha)$
of semi-invariants. This ring is graded by weights $\sigma\in \Z^{Q_0}$, so we have a decomposition
$\SI(Q,\alpha)=\bigoplus_{\sigma\in \N^{Q_0}} \SI(Q,\alpha)_{\sigma}$. For a given weight $\sigma$, we can consider
the subring $\SI(Q,\alpha,\sigma)=\bigoplus_{d=0}^\infty \SI(Q,\alpha)_{d\sigma}$. For any weight $\sigma$, the projective variety $\Proj(\SI(Q,\alpha,\sigma))$, if nonempty,  is a moduli space for the $\alpha$-dimensional representations of the quiver $Q$. See \cite{King} for more details.

In Section~\ref{sec:Quivers} we will give polynomial bounds (in terms of $\alpha,\sigma,Q$) for the generators of $\SI(Q,\alpha,\sigma)$. For the generalized Kronecker quiver $\theta(m)$ and dimension vector $(p,q)$ this gives:
\begin{theorem}\label{pq}
If $\kar(K)=0$, then the invariant ring $K[\M_{p,q}^m]^{\Sl_p\times \Sl_q}$ is generated by invariants of degree $\leq (pq\lcm(p,q))^2$.
\end{theorem}

\subsection{Applications to Algebraic Complexity Theory}
The polynomial degree bound has some interesting applications in Algebraic Complexity Theory. Some applications are related to free skew fields.  
Suppose that $X=(X_1,X_2,\dots,X_m)\in \M_{n,n}^m$ and consider the free skew field $L=K\llangle t_1,t_2,\dots,t_m\rrangle$ generated by $t_1,t_2,\dots,t_m$ (see \cite{Cohn2}.
There is a useful criterion to test invertibility over the skew field (take $Q_0=0$ in \cite[Proposition~7.3]{HW}):

\begin{proposition} \label{inv.skew}
The matrix $A=t_1X_1+t_2X_2+\cdots+t_mX_m\in \M_{n,n}(L)$ is invertible, if and only if there exists a nonnegative integer $d$ and matrices
$T_1,T_2,\dots,T_m\in \M_{d,d}(K)$ such that $X_1\otimes T_1+X_2\otimes T_2+\cdots+X_m\otimes T_m$ is invertible.
\end{proposition}
Various problems in Algebraic Complexity Theory can be reduced to testing whether some linear matrix $A$ is invertible. For this reason, 
Problem 4 in \cite{HW} asks for an upper bound for $\delta(n)$. The polynomial bound for $\delta(n)$ gives us a randomized polynomial time algorithm for determining whether  the linear matrix $A=\sum_{i=1}^m t_i X_i \in \M_{n,n}(L)$ is invertible for infinite fields of arbitrary characteristic.
For $K=\Q$ it  was shown by Garg, Gurvits, Oliviera and Widgerson in \cite{GGOW}  that Gurvits' algorithm in \cite{Gurvits}  can decide invertibility of $A$ in deterministic polynomial time polynomial over $\Q$, without using a polynomial bound on $\delta(n)$ (a weaker bound suffices). 
A similar result can be obtained by combining the results from \cite{IQS} with our polynomial bound for $\delta(n)$.
In Section~\ref{act} we will discuss in more detail, the following consequences from the polynomial bound.
\begin{itemize}
\item {\bf Rational identity testing:} Deciding whether a non-commutative formula computes the zero function can be determined in randomized polynomial time, and in deterministic polynomial time when working over the field $\Q$.
\item{\bf Division-free formulas:} Given a non-commutative polynomial of degree $k$ in $m$ variables which has a formula of size $n$ using additions, multiplications and divisions, then there exists a division-free formula of size $n^{O(\log^2(k)\log(n))}$.
\item{\bf Lower bounds on formula size:} Any formula with divisions computing the non-commutative determinant of degree $n$ must have at least sub-exponential size (in $n$).
\end{itemize}

\subsection{Organisation}
In Section~\ref{lin.subs}, we recall the language of linear subspaces and blow ups and prove Theorem~\ref{main}.
 We prove the degree bounds for invariants defining the null cone and for generating invariants in Section~\ref{upper}. In Section~\ref{lower}, we explain a construction that allows to prove  the lower bound in Theorem~\ref{lower bounds}. In Section~\ref{sec:Quivers} we study degree bounds for quiver semi-invariants,
 and generalize the degree bound for matrix invariants to arbitrary rectangular matrices. In Section~\ref{act} we discuss applications to Algebraic Complexity Theory.

\section{Linear subspaces of matrices and blow ups} \label{lin.subs}
Various properties of an $m$-tuple $X=(X_1,X_2,\dots,X_m)\in \M_{n,n}^m$ only depend on the subspace spanned by 
$X_1,\dots,X_m$. In this section we study such subspaces.

\begin{definition}
Let $\mathcal{X}$ be a linear subspace of $\M_{k,n}$. We define $\rank(\mathcal{X})$ to be the maximal rank among its members, $$\rank(\mathcal{X}) = \max \{ \rank(X) |\  X \in \mathcal{X} \}.$$
\end{definition}

We define tensor blow ups of linear subspaces following \cite{IQS}.

\begin{definition}
Let $\mathcal{X}$ be a linear subspace of $\M_{k,n}$. We define its $(p,q)$ tensor blow up $\mathcal{X}^{\{p,q\}}$ to be $$\mathcal{X} \otimes \M_{p,q} = \Big\{ \sum_{i} X_i \otimes T_i |\  X_i \in \mathcal{X}, T_i \in \M_{p,q} \Big\}$$ viewed as a linear subspace of $\M_{kp,nq}$. We will write ${\mathcal X}^{\{d\}}={\mathcal X}^{\{d,d\}}$.
\end{definition}

%

 In \cite{IQS}, Ivanyos, Qiao and Subrahmanyam prove a regularity lemma (\cite[Lemma~11 and Remark 10]{IQS}) which is crucial for the proof of our main results.
\begin{proposition} [\cite{IQS}] \label{regularity}
If $\mathcal{X}$ is a linear subspace of matrices, then $\rank(\mathcal{X}^{\{d\}})$ is a multiple of $d$.
\end{proposition}

Let us fix $X=(X_1,\dots,X_m)\in \M_{n,n}^m$ and let $\mathcal{X}$ be the span of $X_1,\dots,X_m$. The following lemma is clear.
\begin{lemma}\label{fAexists}
Given a positive integer $d$, the following statements are equivalent:
\begin{enumerate}
\item there exists an $m$-tuple $T\in \M_{d,d}^m$ such that $f_T(X)\neq 0$;
\item $\rank({\mathcal X}^{\{d\}})=dn$.
\end{enumerate}
\end{lemma}

\begin{proof} [Proof of Lemma~\ref{gamman}]
Let $X=(X_1,\dots,X_m)\in \M_{n,n}^m$ and define $\overline{X}=(X_1,\dots,X_m,0)\in \M_{n,n}^{m+1}$.
We have 
$$X\notin {\mathcal N}(n,m)\Leftrightarrow
\mbox{there exists a $d>0$ such that $\rank({\mathcal X}^{\{d\}})=dn$}\Leftrightarrow \overline{X}\notin {\mathcal N}(n,m+1).$$
Suppose that $X\notin {\mathcal N}(n,m)$. Then we have $\overline{X}\not\in {\mathcal N}(n,m+1)$. So there exists $T\in \M_{d,d}^{m+1}$ with $f_T(\overline{X})\neq 0$ and $d\leq \delta(n,m+1)$.
It follows that $\rank({\mathcal X}^{\{d\}})=dn$ so there exists $T\in \M_{d,d}^m$ with $f_T(X)\neq 0$. This proves $\delta(n,m)\leq \delta(n,m+1)$.

If $m>n^2$ and $X\in\M_{n,n}^m\setminus {\mathcal N}(n,m)$, then ${\mathcal X}$ can be spanned by $n^2$ matrices, say $Y_1,\dots,Y_{n^2}$. 
If $Y=(Y_1,\dots,Y_{n^2})$ then there exists $S\in \M_{d,d}^{n^2}$ with $f_S(Y)\neq 0$ and $d\leq \delta(n,n^2)$.  So we have  $\rank({\mathcal X}^{\{d\}})=dn$,
and there exists $T\in \M_{d,d}^m$ with $f_T(X)\neq 0$. This proves that $\delta(n,m)\leq \delta(n,n^2)$. 
\end{proof}

\begin{definition}
We define the function $r:\Z_{\geq 0}\times \Z_{\geq 0}\to \Z_{\geq 0}$ by
$$
r(p,q)=\rank({\mathcal X}^{\{p,q\}}).
$$
\end{definition}
\begin{remark} \label{Z-dense}
Note that the set of all $T=(T_1,\dots,T_m)\in \M_{p,q}^m$ for which $\sum_{i=1}^m X_i\otimes T_i$ has maximal rank $r(p,q)$ is Zariski dense in $\M_{p,q}^m$.
\end{remark}
\begin{lemma} \label{technical}
The function $r$ has the following properties:
\begin{enumerate}
\item $r(p,q+1) \geq r(p,q)$;
\item $r(p+1,q)\geq r(p,q)$;
\item $r(p,q+1) \geq \frac{1}{2}(r(p,q) + r(p,q+2))$;
\item $r(p+1,q)\geq \frac{1}{2}(r(p,q) + r(p+2,q))$;
\item $r(p,q)$ is divisible by $\gcd(p,q)$.
\end{enumerate}
\end{lemma} 

\begin{proof}\ \\
(1) follows from viewing ${\mathcal X}^{\{p,q\}}$ as a subspace of ${\mathcal X}^{\{p,q+1\}}$.

 Now we will prove (3).  Let $T=(T_1,\dots,T_m)\in \M_{p,q+2}^m$.
For a subset $J\subseteq\{1,2,\dots,q+2\}$, let $T_i^J$ be the submatrix where all the columns with index in $J$ are omitted,
and let ${\mathcal Y}_J$ be the column span of $\sum_{i}X_i\otimes T_i^J$. 
If we choose $T$ general enough, then $\sum_{i} X_i\otimes T_i^J$ will have rank $r(p,q+2-|J|)$ for all $J\subseteq \{1,2,\dots,q+2\}$.
We have ${\mathcal Y}_1+{\mathcal Y}_2={\mathcal Y}_{\emptyset}$ and ${\mathcal Y}_{1,2}\subseteq {\mathcal Y}_{1}\cap {\mathcal Y}_2$. It follows that
$$
r(p,q)=\dim {\mathcal Y}_{1,2}\leq \dim {\mathcal Y}_1\cap {\mathcal Y}_2=\dim{\mathcal Y}_1+\dim{\mathcal Y}_2-\dim({\mathcal Y}_{1}+{\mathcal Y}_2)=
2r(p,q+1)-r(p,q+2).
$$

Parts (2) and (4) follow from (1) and (3) respectively by symmetry. 

To see (5), write $p=dp'$ and $q=dq'$. Then we have $\mathcal{X}^{\{p,q\}}=(\mathcal{X}^{\{p',q'\}})^{\{d\}}$ and the result follows from Proposition~\ref{regularity}.
\end{proof}

In the above lemma, parts (1) and (3) give us that $r(p,q)$ is weakly increasing and weakly concave in the second variable, and parts (2) and (4) give the same conclusion for the first variable. 

\begin{corollary} \label{smart}
The function $r(p,q)$ is weakly increasing and weakly concave in either variable. 
\end{corollary}
\begin{lemma}\label{r11}
If $r(1,1)=1$, then we have $r(d,d)=d$ for all $d$. 
\end{lemma}
\begin{proof}
Choose a nonzero matrix $A\in {\mathcal X}$ of rank $1$. Using left and right multiplication with matrices in $\GL_n(K)$ we may assume without loss of generality that
$$
A= \left[ \arraycolsep=5pt\def\arraystretch{1} \begin{array}{cccc}
1 & 0 & \cdots & 0\\
0 & 0 & &  0\\
\vdots & &  \ddots & \vdots\\
0 & 0 & \cdots & 0\end{array}\right].
$$
It is clear that $r(d,d)\geq d$.
If $i>1$,  $j>1$ and $B\in {\mathcal X}$ then $B_{i,j}$ has to be zero, otherwise $tA+B$ will have rank at least $2$ for some $t$.
So ${\mathcal X}$ is contained in 
$$
 \left[ \arraycolsep=5pt\def\arraystretch{1} \begin{array}{cccc}
* & * & \cdots & * \\
* & 0 & \cdots  &  0\\
\vdots & \vdots &  \ddots & \vdots\\
* & 0 & \cdots & 0\end{array}\right].
$$
Because all matrices of ${\mathcal X}$ have rank at most $1$, ${\mathcal B}$ must be contained in the union $W_1\cup W_2$, where
$$
W_1= \left[ \arraycolsep=5pt\def\arraystretch{1} \begin{array}{cccc}
* & 0 & \cdots & 0 \\
* & 0 & \cdots  &  0\\
\vdots & \vdots &  \ddots & \vdots\\
* & 0 & \cdots & 0\end{array}\right]
\mbox{\ \ and\ \  }
W_2
 \left[ \arraycolsep=5pt\def\arraystretch{1} \begin{array}{cccc}
* & * & \cdots & * \\
0 & 0 & \cdots  &  0\\
\vdots & \vdots &  \ddots & \vdots\\
0 & 0 & \cdots & 0\end{array}\right].
$$
Because ${\mathcal X}$ is a subspace, it is entirely  contained in $W_1$ or in $W_2$. Now it is clear that the matrices in ${\mathcal X}^{\{d\}}$ have at most $d$ nonzero columns, or at most $d$ nonzero rows, so $r(d,d)\leq d$.

\end{proof}

\begin{proposition}\label{red}
Let $n \geq 2$, and let $d + 1 \geq n$. If $r(d + 1,d+1)= n(d+1)$, then $r(d,d) = nd$ as well.
\end{proposition}

\begin{proof}
Suppose that $r(d+1,d+1)=n(d+1)$. If $1\leq a\leq d$, then weak concavity implies that
$$
r(d+1,a)\geq \frac{(d+1-a)r(d+1,0)+ar(d+1,d+1)}{d+1}=\frac{an(d+1)}{d+1}=an.
$$
The inequality $r(d+1,a)\leq an$ is clear, so $r(d+1,a)=an$. Similarly, we have $r(a,d+1)=an$.
If $r(1,1)=1$ then we get $r(d+1,d+1)=d+1$ by Lemma~\ref{r11}  which contradicts $r(d+1,d+1)=n(d+1)$. So we have $r(1,1)\geq 2$.
Since $r(p,q)$  is weakly concave in the second variable, we have  $$r(1,d) \geq \frac{(d-1)\cdot r(1,d+1) + 1\cdot r(1,1)}{d} \geq \frac{(d - 1) n +  2 }{d} = n - \frac{n-2}{d} > n-1,$$ 
where the last inequality follows as $d \geq n-1$. Since $r(1,d)$ must be an integer, we have $r(1,d) \geq n$. Now, by the weak concavity in the first variable, we have $$r(d,d) \geq \frac{(d-1)\cdot r(d+1,d) + 1\cdot r(1,d)}{d}\geq \frac{(d-1)nd + n}{d} = nd - n + \frac{n}{d}.$$ 
Note that since $d \geq n-1$, we have $d + \frac{n}{d} > n$ or equivalently that $-n + \frac{n}{d} > -d$. Thus, we have $$r(d,d) \geq nd - n + \frac{n}{d} > d(n-1).$$ Recall that $r(d,d)$ must be a multiple of $d$ by Lemma~\ref{regularity}. Thus $r(d,d) = nd$.

\end{proof}

\begin{proof}[Proof of Theorem~\ref{main}]
Suppose $(X_1,X_2,\dots,X_m) \notin \mathcal{N}(n,m)$. By Lemma~\ref{fAexists}, $r(d,d)=dn$ for some $d$.
Without loss of generality, we can assume $d \geq n$. By repeated application of Proposition~\ref{red}, we conclude that $r(n-1,n-1) = n(n-1)$. So, again by Lemma~\ref{fAexists}, there exists 
an $m$-tuple $T=(T_1,\dots,T_m) \in \M_{n-1,n-1}^m$ such that $f_T(X)\neq 0$.
\end{proof}

\section{Degree bounds on generating invariants}\label{upper}
Suppose that the base field $K$ has characteristic 0,  $G$  is a connected semisimple group and $V$ is a representation of $G$. A homogeneous system of parameters for the invariant ring $K[V]^G$ is a set of homogeneous invariants
$f_1,f_2,\dots,f_r$ such that $f_1,f_2,\dots,f_r$ are algebraically independent and $K[V]^G$ is a finitely generated $K[f_1,\dots,f_r]$-module. 
The ring $K[V]^G$ is a finitely generated $K[f_1,\dots,f_r]$-module if and only if the zero set  of $f_1,\dots,f_r$ is the null cone (see\cite{Hilbert2}).
\begin{definition}
For a representation $V$ of a connected semisimple group $G$, $\beta(\K[V]^G)$ is defined as the smallest integer $d$ such that invariants of degree $\leq d$ generate the ring of invariants $\K[V]^G$. 
\end{definition}
Using the homogeneous system of parameters in Corollary~\ref{hsop}, we can get a bound for the generating invariants (see \cite{Popov1,Popov2} and~\cite[Corollary~2.6.3]{DK}):

\begin{proposition}\label{Popov}
Suppose $V$ is a representation of a connected semisimple group $G$. Let $f_1,f_2,\dots,f_r$ be a homogeneous system of parameters for $\K[V]^G$, and let $d_i = \deg(f_i)$. Then  $$\beta(\K[V]^G) \leq \max \{d_1+d_2 + \dots d_r - r,d_1,d_2,\dots,d_r \}.$$ 
\end{proposition} 

We go back to the special case where  $V=\M_{n,n}^m$,  $G=\SL$, and $\beta(n,m)=\beta(K[V]^G)$.

\begin{corollary} \label{hsop}
Let $n \geq 2$, and let $r$ be the Krull dimension of $R(n,m)$. Then there exist $r$ invariants of degree $n^2 - n$ that form a homogeneous system of parameters. 
\end{corollary}

\begin{proof}
By Theorem~\ref{main}, the invariants of degree $n^2 - n$ define the null cone. We apply the Noether normalization lemma (see \cite[Lemma~2.4.7]{DK}) to conclude that there exists $r$ invariants of degree $n^2 - n$ that form a homogeneous system of parameters.
\end{proof}

\begin{proof}[Proof of Theorem~\ref{deg.bds}]
For $n \geq 2$, we apply the  above proposition to the left-right action of $\SL$ on $n^2$-tuples of matrices using the homogeneous system of parameters from Corollary~\ref{hsop} to get $$
\beta(n,m) \leq  r (n^2 - n) - r = r(n^2 - n - 1) \leq mn^2(n^2 - n -1) < mn^4.$$ It is clear that $\beta(R(1,m)) = 1$, so we have $\beta(R(n,m)) \leq mn^4$ for all $n$ and $m$.
\end{proof}

\section{Lower bounds for $\gamma(n)$ and $\delta(n)$} \label{lower}
In this section we prove  Theorem~\ref{lower bounds}.
Let $ A =  t_1X_1 + t_2X_2 + \dots + t_{m}X_{m}$ be an $n \times n$ linear matrix. The $(i,j)^{th}$ entry of $A$ is a linear function in the indeterminates $t_k$'s with coefficients in $\K$. In fact if $c_k \in \K$ is the $(i,j)^{th}$ entry of $X_k$, then the $(i,j)^{th}$ entry of $A$ is given by $$A_{i,j} =  \sum_{k=1}^m c_k t_k.$$ For $p \times p$ matrices $T_1,T_2,\dots,T_{m}$, observe that the expression  $\sum_{k=1}^{m} X_k \otimes T_k$ is an $n \times n$ block matrix and the size of each block is $p \times p$.  Moreover, the $(i,j)^{th}$ block is $$\sum_{k=1}^m c_k T_k.$$

\begin{remark} \label{substitution}
In effect $ \sum_{k=1}^{m} X_k \otimes T_k$ is simply the block matrix obtained by substituting the $T_k$ for $t_k$ in the linear matrix $A$. 
\end{remark}
\begin{lemma} \label{assumel}
If there exist $k\times k$ matrices $T_1,T_2,\dots,T_k$ such that $X_1\otimes T_1+\cdots+X_k\otimes T_k$ is invertible, 
then there exists $k\times k$ matrices $S_2,S_3,\dots,S_k$ such that $X_1\otimes I+X_2\otimes S_2+\cdots+X_k\otimes S_k$ is invertible.
\end{lemma}
\begin{proof}
If there are exists $T_1,T_2,\dots,T_k$ such that $\sum_{i=1}^m X_i\otimes T_i$ is invertible, then this matrix will be invertible for general choices of $T_1,\dots,T_k$.
In particular, without loss of generality we may assume that $T_1$ invertible.  If we set $S_i=T_1^{-1}T_i$ for $i\geq 2$, then we have
$$
(I\otimes T_1)^{-1}\sum_{i=1}^m X_i\otimes T_i=X_1\otimes I+X_2\otimes S_2+\cdots+X_k\otimes S_k
$$
is invertible.
\end{proof}

Given the remark and lemma above, we now state a straightforward lemma which follows from the definition of $\delta(n)$.

\begin{lemma} \label{obvious}
Suppose we have $X=(X_1,\dots,X_m)\in \M_{n,n}^m$ and suppose that  the linear matrix $A=\sum_{i=1}^m t_iX_i$  has the properties:
\begin{enumerate}
 \item For any $k < d$, substituting $t_1=I$ and substituting any $k \times k$ matrices for the indeterminates $t_2,t_3,\dots,t_m$ gives us a singular matrix;
 \item there exists a particular substitution of $d \times d$ matrices for $t_1,t_2,\dots,t_m$ which gives a non-singular matrix.
\end{enumerate}
Then we have  $\delta(n,m) \geq d$ and $\delta(n)\geq d$.
\end{lemma}

One can use the procedure in \cite[Section~6]{HW} to construct a linear matrix in which the top right corner entry of its inverse (over the skew field) is any desired rational expression. For any $d$, we can find non-trivial rational expressions which are not defined for matrices of size $< d$, such as taking the inverse of the famous Amitsur-Levitzki polynomial (see~\cite{AL}). However, the size of the linear matrix becomes very large giving us very weak bounds. 

To find better bounds, we want to keep the size of $n$ as small as possible, and we present the most efficient that we are able to find. We make use of the Cayley-Hamilton theorem, which says that a matrix satisfies its characteristic polynomial. For the sake of clarity, we discuss it in detail for $d = 3$, and then describe the general construction.   

For this construction, $A,B,$ and $C$ will denote arbitrary $k \times k$ matrices, and $I$ will denote the identity matrix of size $k \times k$. First consider the block matrix $$N_3 = \left[
\arraycolsep=5pt\def\arraystretch{1}\begin{array}{ccc}
A^2B & AB & B \\
A^2C & AC & C \\
A^2 & A & I \end{array} \right].$$ 

If $k \leq 2$, then the characteristic polynomial of $A$ gives us a linear dependency in the columns. For example, if $k=2$ and the characteristic polynomial of $A$ is $t^2+at+b$,
then we have
$$
\left[\arraycolsep=5pt\def\arraystretch{1}\begin{array}{ccc}
A^2B & AB & B \\
A^2C & AC & C \\
A^2 & A & I \end{array} \right]
\left[
\arraycolsep=5pt\def\arraystretch{1}\begin{array}{c}
I\\
aI\\
bI
\end{array} \right]=0.
$$

However, if we pick \begin{equation} \label{ABC}
A = \left[
\arraycolsep=5pt\def\arraystretch{1}\begin{array}{ccc}
\lambda_1& 0 & 0 \\
0 & \lambda_2 & 0 \\
0 & 0 & \lambda_3 \end{array} \right], B = \left[
\arraycolsep=5pt\def\arraystretch{1}\begin{array}{ccc}
0& 0 &1 \\
1 & 0 &0 \\
0 & 1 & 0
 \end{array} \right] \text{ and } C = \left[
\arraycolsep=5pt\def\arraystretch{1}\begin{array}{ccc}
0& 1 & 0 \\
0 & 0 &1 \\
1 & 0 & 0
 \end{array} \right],\end{equation} with the $\lambda_i$ pairwise distinct, then $$N_3 = \left[ \arraycolsep=5pt\def\arraystretch{1} \begin{array}{ccc|ccc|ccc}
0 &  0 & \lambda_1^2 & 0 & 0 & \lambda_1 & 0 & 0 & 1\\
\lambda_2^2 &  0 & 0 & \lambda_2 & 0 & 0 & 1 & 0 & 0 \\
0 &  \lambda_3^2 & 0 & 0 & \lambda_3 & 0 & 0 & 1 & 0 \\ \hline
0 &  \lambda_1^2 & 0 & 0 & \lambda_1& 0 & 0 & 1 & 0 \\ 
0 &  0  & \lambda_2^2 & 0 & 0 & \lambda_2 & 0 & 0 & 1\\
\lambda_3^2 &  0 & 0 & \lambda_3 & 0 & 0 & 1 & 0 & 0 \\\hline
\lambda_1^2 &  0 & 0 & \lambda_ 1& 0 & 0 & 1 & 0 & 0 \\
0 &  \lambda_2^2 & 0 & 0 & \lambda_2 & 0 & 0 & 1 & 0 \\ 
0 &  0 &  \lambda_3^2 & 0 & 0 & \lambda_3 & 0 & 0 & 1\\
\end{array} \right]. $$

Permuting the rows of $N_3$, we get $$\left[ \arraycolsep=5pt\def\arraystretch{1} \begin{array}{ccc|ccc|ccc}

\lambda_1^2 &  0 & 0 & \lambda_1 & 0 & 0 & 1 & 0 & 0 \\
\lambda_2^2 &  0 & 0 & \lambda_2 & 0 & 0 & 1 & 0 & 0 \\
\lambda_3^2 &  0 & 0 & \lambda_ 3& 0 & 0 & 1 & 0 & 0 \\ \hline
0 &  \lambda_1^2 & 0 & 0 & \lambda_1 & 0 & 0 & 1 & 0 \\ 
0 &  \lambda_2^2 & 0 & 0 & \lambda_2& 0 & 0 & 1 & 0 \\ 
0 &  \lambda_3^2 & 0 & 0 & \lambda_3 & 0 & 0 & 1 & 0 \\ \hline
0 &  0 & \lambda_1^2 & 0 & 0 & \lambda_1 & 0 & 0 & 1\\ 
0 &  0  & \lambda_2^2 & 0 & 0 & \lambda_2 & 0 & 0 & 1\\ 
0 &  0 &  \lambda_3^2 & 0 & 0 & \lambda_3 & 0 & 0 & 1\\
\end{array} \right].$$

Then permuting the columns, we get $$\left[ \arraycolsep=5pt\def\arraystretch{1} \begin{array}{ccc|ccc|ccc}

\lambda_1^2 &  \lambda_1 & 1 & & 0 & 0 & 0 & 0 & 0 \\
\lambda_2^2 &  \lambda_2 & 1 & 0 & 0 & 0 & 0 & 0 & 0 \\
\lambda_3^2 & \lambda_3 & 1 & 0 & 0 & 0 & 0 & 0 & 0 \\ \hline
0 &   0 & 0 & \lambda_1^2 & \lambda_1 & 1 & 0 & 0 & 0 \\ 
0 &   0 & 0 &\lambda_2^2 & \lambda_2& 1 & 0 & 0 & 0 \\ 
0 &   0 & 0 &\lambda_3^2 & \lambda_3  & 1 & 0 & 0 & 0 \\ \hline
0 &  0 &  0 & 0  & 0 & 0 & \lambda_1^2 & \lambda_1 & 1\\ 
0 &  0  & 0 & 0 &  0 & 0 & \lambda_2^2 &\lambda_2 & 1\\ 
0 &  0 &   0 & 0 &  0 & 0 & \lambda_3^2 & \lambda_3 &1\\
\end{array} \right],$$ and hence $N_3$ is non-singular as the $\lambda_i$ are pairwise distinct. The one problem with using this directly is the non-linearity of the entries in $N_3$. To fix this, we consider the $8 \times 8 $ block matrix $$F_3 = \left[ \arraycolsep=5pt\def\arraystretch{1.2}\begin{array}{ccccc|ccc}
I & \multicolumn{1}{c|}{}    &      &        &   & B &      &  \\ 
-A  & \multicolumn{1}{c|}{I}   &      &       &    &    &  B & \\
     & \multicolumn{1}{c|}{-A}  &      &       &    &    &     & B\\ \cline{1-4} \cline{6-8}
     &      &  \multicolumn{1}{|c}{I}  &   \multicolumn{1}{c|}{}   &     &C &   & \\
     &      & \multicolumn{1}{|c}{-A} &    \multicolumn{1}{c|}{I}  &     &    & C & \\
     &      &  \multicolumn{1}{|c}{}   &   \multicolumn{1}{c|}{-A} &     &    &     & C \\ \cline{3-8}
     &       &    &        & \multicolumn{1}{|c|}{I}   & A&  & \\
     &       &     &       & \multicolumn{1}{|c|}{-A} &    & A & I\\
\end{array} \right].$$

The invertibility of such a block matrix is unaffected by adding left multiplied block rows to other block rows, and by adding right multiplied block columns to other block columns. We left multiply the first block row by $A$ and add it to the second block row. Then we left multiply the second block row by $A$ and add it to the third block row. Focusing on the top three block rows, we have transformed
$$ \left[ \arraycolsep=5pt\def\arraystretch{1}\begin{array}{cccccccc}
I    &   \multicolumn{1}{c|}{}   &      &        &   \multicolumn{1}{c|}{}& B &      &  \\ 
-A  &\multicolumn{1}{c|}{I}   &      &       &       \multicolumn{1}{c|}{} & & B & \\
     & \multicolumn{1}{c|}{-A}  &      &       &       \multicolumn{1}{c|}{} &   &  & B\\
\end{array} \right] \longrightarrow \left[ \arraycolsep=5pt\def\arraystretch{1}\begin{array}{cccccccc}
I    & \multicolumn{1}{c|}{}     &      &        &   \multicolumn{1}{c|}{} & B & &        \\ 
  & \multicolumn{1}{c|}{I}   &      &        &    \multicolumn{1}{c|}{}    & AB  &  B & \\
     &  \multicolumn{1}{c|}{}&      &       &     \multicolumn{1}{c|}{}   & A^2B &  AB  & B\\
\end{array} \right].$$
     
We can also right multiply block columns by a matrix and add them to other block columns. So, we can further transform the top 3 block rows to $$ \left[ \arraycolsep=5pt\def\arraystretch{1}\begin{array}{cccccccc}
I    &   \multicolumn{1}{c|}{}   &      &        &  \multicolumn{1}{c|}{} &&  &        \\ 
  & \multicolumn{1}{c|}{I}   &      &        &     \multicolumn{1}{c|}{}  & &  &    \\ 
     & \multicolumn{1}{c|}{} &      &       &     \multicolumn{1}{c|}{} & A^2B &  AB  & B\\ \end{array} \right].$$ 
      
Notice that these transformations do not affect the rest of the block rows in $F_3$. A similiar procedure for the next $3$ block rows, and then for the last two block rows shows that the invertibility of $F_3$ is equivalent to the invertibility of   
$$\left[ \arraycolsep=5pt\def\arraystretch{1} \begin{array}{cccccccc}
I    &      &      &        &   &&  &        \\ 
  & I   &      &        &       & &  &    \\ 
       &  &      &         &    & A^2B &  AB  & B\\ \hline
   &      &    I  &        &   &&  &        \\ 
  &   &      &  I      &       & &  &    \\ 
     &  &            &    &    & A^2C &  AC  & C\\ \hline
  &    &      &        &   I    & &  &    \\ 
     &  &      &         &    & A^2 &  A  & I\\ 
 \end{array} \right],$$ which is then equivalent to the invertibility of $N_3$. 
 
Thus if $A,B,$ and $C$ are square matrices of size $\leq 2$, then $F_3$ is always singular. However, there exists a particular choice of $3 \times 3$ matrices, i.e, (\ref{ABC}), for which $F_3$ is invertible. We can write $F_3$ as $X_1\otimes I+X_2\otimes A+X_3\otimes B+X_4\otimes C$
and consider $X=(X_1,X_2,X_3,X_4)\in \M_{8,8}^4\setminus {\mathcal N}(8,4).$

 The above discussion shows that $X$ satisfies the conditions of 
 Lemma~\ref{obvious} for $n = 8, d= 3$, and so we get $\delta(8) \geq 3$.

For the general construction, consider $$N_d =  \left[ \arraycolsep=5pt\def\arraystretch{1} \begin{array}{ccccc}
A^{d-1}B_1   & A^{d-2}B_1     &  \cdots     & B_1  \\
 A^{d-1}B_2  &   \ddots               &                 & B_2  \\
 \vdots      &       &    \ddots  &\vdots   \\
A^{d-1}B_{d-1}  &       \cdots       & \cdots               & B_{d-1} \\
A^{d-1} & \cdots & \cdots & I \\
 \end{array} \right],$$
     
 where $A,B_i$ are taken to be arbitrary $k \times k$ matrices. If $k < d$, then the characteristic polynomial of $A$ gives a linear dependency on the columns. On the other hand, choose $A$ to be a diagonal $d\times d$ matrix with pairwise distinct diagonal entries $\lambda_1,\lambda_2, \dots,\lambda_d$, choose $B_1$ to be the permutation matrix corresponding to the long cycle in the symmetric group on $d$ letters, and choose $B_i = B_1^i$. Similar to the case of $N_3$, we can permute the rows and columns to transform it into a block diagonal matrix, where each diagonal block is a Vandermonde matrix, and hence invertible. 

Similiar to the construction of $F_3$, we construct $F_d$ and this has size $d^2 - 1 \times d^2 -1$. To do this, we define an $n \times n-1$ block matrix $\mathcal{P}_n(A)$ and an $n-1 \times n$ block matrix $\mathcal{Q}_n(A)$ by  $$\mathcal{P}_n(A) =  \left[ \arraycolsep=5pt\def\arraystretch{1} \begin{array}{ccccc}
I & & & \\
-A & I & & \\ 
&   \ddots & \ddots  &\\
&   & -A & I \\
&   & & -A \\
\end{array} \right], \text{ and } \mathcal{Q}_n(A) = \left[ \arraycolsep=5pt\def\arraystretch{1} \begin{array}{cccccc}
 A & 0 & & & \\
   & A & \ddots & &  \\
   & & \ddots & 0 & \\
   & &           &  A & I\\ 
\end{array} \right]. $$

Notice that $F_3$ is just the block matrix $$ \left[ \arraycolsep=5pt\def\arraystretch{1} \begin{array}{cccc}
\mathcal{P}_3(A) & & & I_3 \otimes B \\
& \mathcal{P}_3(A) & & I_3 \otimes C \\
& & \mathcal{P}_2(A) & \mathcal{Q}_3(A) \\ \end{array} \right], $$ where $I_3$ denotes the identity matrix of size $3 \times 3$.  Now we define $$F_d = \left[ \arraycolsep=5pt\def\arraystretch{1} \begin{array}{cccccc}
\mathcal{P}_d(A) & & & & & I_d \otimes B_1 \\
& \mathcal{P}_d(A) & & & & I_d \otimes B_2 \\
& & \ddots  &&& \vdots \\
& & & \mathcal{P}_{d}(A) & & I_d \otimes B_{d-1} \\
& & & & \mathcal{P}_{d-1}(A) & \mathcal{Q}_d(A)
\end{array} \right], $$ where $I_d$ denotes the identity matrix of size $d \times d$. We can write 
$$F_d=X_1\otimes I+X_2\otimes A+X_3\otimes B_1+\cdots +X_{d+1}\otimes B_{d-1}$$ and
we  consider 
$$X=(X_1,X_2,\dots,X_{d+1})\in \M_{d^2-1,d^2-1}^{d+1}\setminus {\mathcal N}(d^2-1,d+1).$$

A similar argument as in the case of $d =3$, shows that the invertibility of $F_d$ is equivalent to the invertibilty of $N_d$. Thus, by Lemma~\ref{obvious}, we have $\delta(d^2 - 1,d+1) \geq d$
and therefore $\delta(d^2-1)\geq d$.  Replacing $d^2 - 1$ by $n$, we get  $\delta(n)\geq \lfloor \sqrt{n+1}\rfloor$ and $\gamma(n)=n\delta(n)\geq n\lfloor \sqrt{n+1}\rfloor$.

\section{Generating invariants for quiver representations}\label{sec:Quivers}
In this section, we generalize our degree bounds for matrix invariants to quiver representations. We start by introducing the common terminology.
A quiver is just a directed graph. Formally a quiver is a pair
$Q=(Q_0,Q_1)$, where $Q_0$ is a finite set of vertices and $Q_1$ is a finite set of arrows. For an arrow $a\in Q_1$ we denote its head and tail by $ha$ and $ta$ respectively. A path of length $k$ is a sequence $p=a_ka_{k-1}\cdots a_1$ where $a_1,\dots,a_k$ are arrows such that $ha_{i-1}=ta_i$ for $i=2,3,\dots k$. The head and tail of the path are defined by $hp=ha_{k}$ and $tp=ta_1$ respectively.
For every vertex $x\in Q_0$ we also have a trivial path $\varepsilon_x$ of length $0$ such that $h\varepsilon_x=t\varepsilon_x=x$.
A cyclic path is a path $p$ of positive length such that $hp=tp$. We will assume that $Q$ has no cyclic paths.

We fix an infinite field $K$. A representation $V$ of $Q$ over $K$ is a collection of finite dimensional $K$-vector spaces $V(x)$, $x\in Q_0$ together with a collection of $K$-linear maps $V(a):V(ta)\to V(ha)$, $a\in Q_1$. The dimension vector of $V$ is the function $\alpha:Q_0\to \N$ such that
$\alpha(x)=\dim V(x)$ for all $x \in Q_0$. If $p=a_ka_{k-1}\cdots a_1$ is a path, then we define
$$
V(p)=V(a_k)V(a_{k-1})\cdots V(a_1):V(tp)\to V(hp).
$$
We define $V(\varepsilon_x)$ is the identity map from $V(x)$ to itself.
For a dimension vector $\alpha\in \N^{Q_0}$, we define its representation space by:
$$\Rep(Q,\alpha)=\prod_{a\in Q_1}\M_{\alpha(ha),\alpha(ta)}.$$ 
 If $V$ is a representation with dimension vector $\alpha$
and we identify $V(x)\cong K^{\alpha(x)}$ for all $x$, then $V$ can be viewed as an element of $\Rep(Q,\alpha)$.
Consider the group $\GL(\alpha)=\prod_{x\in Q_0} \GL_{\alpha(x)}$ and its subgroup $\Sl(\alpha)=\prod_{x\in Q_0}\Sl_{\alpha(x)}$.
The group $\GL(\alpha)$ 
acts on $\Rep(Q,\alpha)$ by:
$$
(A(x)\mid x\in Q_0)\cdot (V(a)\mid a\in Q_1)=(A(ha)V(a)A(ta)^{-1}\mid a\in Q_1).
$$
For $V\in \Rep(Q,\alpha)$, choosing a different basis means acting by the group $\GL(\alpha)$. The $\GL(\alpha)$-orbits in $\Rep(Q,\alpha)$ correspond to isomorphism classes of representations of dimension $\alpha$.
The group $\GL(\alpha)$ also acts (on the left)  on the ring $K[\Rep(Q,\alpha)]$ of polynomial functions on $\Rep(Q,\alpha)$ by
$$
A\cdot f(V)=f(A^{-1}\cdot V)
$$
where $f\in K[\Rep(Q,\alpha)]$, $V\in \Rep(Q,\alpha)$ and $A\in \GL(\alpha)$.

The invariant ring $\SI(Q,\alpha)=K[\Rep(Q,\alpha)]^{\Sl(\alpha)}$ is called the ring of semi-invariants.
A multiplicative character of the group $\GL_\alpha$ is of the form
$$
\chi_\sigma:(A(x)\mid x\in Q_0)\in \GL_\alpha\mapsto \prod_{x\in Q_0}\det(A(x))^{\sigma(x)}\in K^\star,
$$
where $\sigma:Q_0\to \Z$ is called the weight of the character $\chi_\sigma$. Define 
$$\SI(Q,\alpha)_\sigma=\{f\in K[\Rep(Q,\alpha)]\mid \forall A\in \GL(\alpha)\ A\cdot f=\chi_\sigma(A) f\}.$$
Then we have $\SI(Q,\alpha)=\bigoplus_{\sigma}\SI(Q,\alpha)_\sigma$. If $\sigma\cdot \alpha=\sum_{x\in Q_0}\sigma(x)\alpha(x)\neq 0$, then $\SI(Q,\alpha)_\sigma=0$. Assume that $\sigma\cdot \alpha=0$. We can write $\sigma=\sigma_+-\sigma_-$ where $\sigma_+(x)=\max\{\sigma(x),0\}$ and $\sigma_-(x)=\max\{-\sigma(x),0\}$. Define $n=\sigma_+\cdot \alpha=\sigma_-\cdot \alpha$.

Now we define a linear matrix $n\times n$
$$A:\bigoplus_{x\in Q_0}V(x)^{\sigma_+(x)}\to \bigoplus_{x\in Q_0}V(x)^{\sigma_-(x)}
$$
where each block $\Hom(V(x),V(y))$ is of the form $t_1V(p_1)+\cdots+t_rV(p_r)$ where $t_1,t_2,\dots,t_r$ are indeterminates and $p_1,p_2,\dots,p_r$ are all paths from $x$ to $y$.
We use different indeterminates for the different blocks, so the  linear matrix has $m=\sum_{x\in Q_0}\sum_{y\in Q_0} \sigma_+(x)b_{x,y}\sigma_-(y)$ indeterminates where $b_{x,y}$ is the number of paths from $x$ to $y$. We can write $A=t_1X_1+\cdots+t_mX_m$ with $X_1,\dots,X_m\in \M_{n,n}$.
We have the following result (see~\cite[Corollary 3]{DW}, \cite{DZ} and  \cite{SVd}).
\begin{theorem}
The space $\SI(Q,\alpha)_\sigma$ is spanned by $\det(t_1X_1+\cdots+t_mX_m)$ with $t_1,\dots,t_m\in K$.
\end{theorem}
\begin{corollary}
For any positive integer $d$, the space $\SI(Q,\alpha)_{d\sigma}$ is spanned by $\det(X_1\otimes T_1+\cdots+X_m\otimes T_m)$ with $T_1,\dots,T_m\in \M_{d,d}$.
\end{corollary}
\begin{proof}
This follows from the construction for $d\sigma$ instead of $\sigma$.
\end{proof}

\begin{corollary}
We have a surjective ring homomorphism $\psi: K[\M_{n,n}^m]^{\SL}\to \SI(Q,\alpha)$ which sends homogeneous elements of degree $dn$ into $\SI(Q,\alpha)_{d\sigma}$.
\end{corollary}
A representation $V\in \Rep(Q,\alpha)$ is called $\sigma$-semistable if there exists an semi-invariant $f\in \SI(Q,\alpha)_{d\sigma}$ with $f(V)\neq 0$ (see~\cite{King}).
\begin{corollary}\label{existssemi}
If $V$ is $\sigma$-semistable, $n=\sum_{x\in Q_0}\sigma_+(x)\alpha(x)$ and $d\geq n-1$, then there exists an semi-invariant $f\in \SI(Q,\alpha)_{d\sigma}$ with $f(V)\neq 0$.
\end{corollary}
The ring $\SI(Q,\alpha,\sigma)=\bigoplus_{d\sigma}\SI(Q,\alpha)_{d\sigma}$ is graded, where $\SI(Q,\alpha)_{d\sigma}$ is the degree $d$ part.
\begin{corollary}
The ring $\SI(Q,\alpha,\sigma)$ is generated in degree $\leq n^5$ where $n=\sum_{x\in Q_0}\sigma_+(x)\alpha(x)$.
\end{corollary}
 
 Let us consider again the Kronecker quiver $\theta(m)$, with dimension vector $\alpha=(p,q)$. Let $e=\gcd(p,q)$ and write $p=p'e$, $q=q'e$. Define $\sigma=(q',-p')$.  We have $n=pq'=p'q=pq/e=pq/\gcd(p,q)=\lcm(p,q)$.
 We have $\SI(Q,\alpha)=\bigoplus_{d=0}^\infty \SI(Q,\alpha)_{d\sigma}=K[\M_{p,q}^m]^{\Sl_p\times \Sl_q}$. The null cone in this case is the set of representations that are not $\sigma$-semistable  (see~\cite{King}).  From Corollary~\ref{existssemi} follows:
 \begin{corollary}
If $d\geq \lcm(p,q)-1$, then the null cone the action of $\Sl_p\times \Sl_q$ in $\M_{p,q}^m$, is defined by invariants of degree $\leq \lcm(p,q)d$.
 \end{corollary}
 \begin{proof}[Proof of Theorem~\ref{pq}]
 Invariants of degree $\lcm(p,q)^2$ define the null-cone. By the Noether normalization lemma, we can find a homogeneous system of parameters in degree $\lcm(p,q)^2$. The number of elements in the homogeneous system of parameters is  $\dim K[\M_{p,q}^m]^{\Sl_p\times \Sl_q}\leq mpq$.
 So by Proposition~\ref{Popov}, the ring $K[\M_{p,q}^m]^{\Sl_p\times\Sl_q}$ is generated in degree $\leq mpq(\lcm(p,q))^2$. Again by a theorem of Weyl
 (see \cite[Section 7.1, Theorem A]{KP}), we may assume that $m\leq pq$.
 \end{proof}

 \section{Applications to algebraic complexity}\label{act}
We have already seen in the introduction that our results give a deterministic algorithm for the invertibility of a linear matrix over $\Q$. In \cite{HW}, Hrube\v s and Wigderson study non-commutative arithmetic circuits, and they comment that perhaps the most important problem that their work suggests is to find a good bound for $\delta(n)$. We describe the consequences of our bound for $\delta(n)$ in algebraic complexity. 

A non-commutative arithmetic circuit is a directed acyclic graph, whose vertices are called gates. Gates of in-degree $0$ are elements of $\K$ or variables $t_i$. The other allowed gates are inverse, addition and multiplication gates of in-degrees 1, 2 and 2 respectively. The edges going into an multiplication gate are labelled left and right to indicate the order of multiplication. 
A formula is a circuit, where every node has out-degree at most $1$. The number of gates in a circuit is called its size. 
A non-commutative rational function over $\K$ in the variables $t_1,t_2,\dots,t_m$ is an element of the skew field $L = \K \llangle t_1,t_2,\dots,t_m\rrangle$. A circuit $\Phi$ in the variables $t_1,t_2,\dots,t_m$ computes a non-commutative rational function for each output gate. 
We denote by $\widehat\Phi(T)$ the evaluation of $\Phi$ at $T = (T_1,T_2,\dots,T_m) \in \M_{p,p}^m$. In the process of evaluation, if the input of an inverse gate is not invertible, then $\widehat{\Phi}(T)$ is undefined. $\Phi$ is called a correct circuit if $\widehat\Phi(T)$ is defined for some $T$. 
For further details, we refer to \cite{HW}.

\begin{definition}
The number $w(n)$ is the smallest integer $d$ such that for every correct formula $\Phi$ of size $n$ (in the variables $t_1,t_2,\dots,t_m$), there exists $T \in \M_{p,p}^m$ with $p \leq d$ such that $\widehat\Phi(T)$ is defined. 
\end{definition}

We have $w(n) \leq \delta(n^2 + n)$ by \cite[Proposition~7.6]{HW}. However, due to the nature of our results, we can do even better. 

\begin{proposition}
We have $w(n) \leq 2n-1$. 
\end{proposition}

\begin{proof}
Given a formula $\Phi$ of size $n$, for each gate $v$, we denote by $\Phi_v$ the sub-formula rooted at $\Phi$. We can construct linear matrices $A_{\Phi_v}$ (in the variables $t_1,t_2,\dots,t_m$) such that $\Phi$ is a correct formula if and only if $A_{\Phi_v}$ is invertible (over the skew field $L$) for all $v$ (see \cite[Corollary~7.2]{HW}). Moreover the matrices $A_{\Phi_v}$ have size $\leq 2n$ (see \cite[Theorem~2.5]{HW}). 

Assume $\Phi$ is a correct formula. Since $A_{\Phi_v} = X_0 + t_1X_1 + t_2X_2 + \dots + t_mX_m$ is invertible, for some $k$ there exists $T = (T_1,T_2, \dots, T_m) \in \M_{k,k}^m$ such that $A_{\Phi_v}(T) = X_0 \otimes I + \sum_{i=1}^m X_i \otimes T_i$ is invertible (see Proposition~\ref{inv.skew} and Lemma~\ref{assumel}). We can assume $k = 2n-1$ by Proposition~\ref{red}. In fact, by Remark~\ref{Z-dense} a general $m$-tuple $T \in \M_{2n-1,2n-1}^m$ suffices. Hence for a sufficiently general $T \in \M_{2n-1,2n-1}^m$, all the $A_{\Phi_v}(T)$ are simultaneously invertible and hence $\widehat{\Phi}(T)$ is defined (see \cite[Proposition~7.1]{HW}).
\end{proof}

\subsection*{Rational identity testing} 
Deciding whether a non-commutative formula computes the zero function is called the rational identity testing problem. Hrube\v s and Wigderson give a randomized algorithm for rational identity testing whose run time is polynomial in $n$ and $w(n)$. See \cite[Section~7]{HW} for the details. Thus the above bound on $w(n)$ gives a polynomial time randomized algorithm for rational identity testing for infinite fields in arbitrary characteristic.

As observed in \cite{GGOW}, we have a deterministic polynomial time algorithm if $\K = \Q$, since the invertibility of linear matrices can be decided in deterministic polynomial time.

\subsection*{Eliminating inverse gates} Let $f$ be a non-commutative polynomial in $\K\langle t_1,t_2,\dots,t_m \rangle$ of degree $k$, which can be computed by a formula of size $n$. Then $f$ can be computed by a formula of size $n^{O(\log^2(k) \log(n) )}$ without inverse gates.  (see \cite[Corollary~8.4]{HW}).

\subsection*{Lower bounds on formula size} Problem 1 in \cite{HW} asks for an explicit  family of non-commutative polynomials which cannot be computed by a polynomial size formula with divisions. We give an answer to this problem. In \cite{Nisan}, it was proved that any formula without divisions computing the non-commutative determinant (or permanent) of degree $k$ 
must have size $2^{\Omega(k)}$. To find the size of a formula that allows divisions, we use our bound for eliminating inverse gates, and solve $2^{\Omega(k)}=n^{O(\log^2(k) \log(n) )}$ for $n$. This
shows that  any formula with divisions computing the non-commutative determinant (or permanent) of degree $k$
has size $2^{\Omega(\sqrt{k}/\log(k) )}$.

\subsection*{Acknowledgements}
The authors like to thank Avi Widgerson and Ketan Mulmuley for helpful discussions. 
We  would like to thank the authors of \cite{IQS,HW,GGOW,GCT} for sending early versions of their papers.

\end{document}